\documentclass[11pt,a4paper]{amsart}
\usepackage{amssymb,amsmath,amsthm}

\usepackage{graphicx}

\usepackage{epsf}

\newtheorem{theorem}{Theorem}[section]
\newtheorem{lemma}{Lemma}[section]
\newtheorem{corollary}{Corollary}[section]

\newtheorem{prop}{Proposition}[section]

\newcommand{\de}{\delta}
\newcommand{\al}{\alpha}
\newcommand{\be}{\beta}
\newcommand{\g}{\gamma}
\newcommand{\G}{\Gamma}
\newcommand{\la}{\lambda}
\newcommand{\La}{\Lambda}
\newcommand{\tr}{\triangle}
\newcommand{\eps}{\varepsilon}
\newcommand{\E}{\mathbb{E}}
\newcommand{\z}{{\mathbb{Z}^2}}
\newcommand{\rr}{\mathbb{R}^2}
\newcommand{\PP}{\mathbb{P}}
\newcommand{\A}{\mathrm{A}}
\newcommand{\cc}{\mathcal{C}}
\newcommand{\dist}{\textrm{dist}}
\newcommand{\conv}{\textrm{conv}}
\newcommand{\J}{\mathcal{J}}
\newcommand{\Ge}{\mathcal{G}}

\numberwithin{equation}{section}

\begin{document}

\title{Longest convex chains}

\author{Gergely Ambrus, Imre B\'ar\'any}



\begin{abstract}
Assume $X_n$ is a random sample of $n$ uniform, independent points
from a triangle $T$. The longest convex chain, $Y$, of $X_n$ is
defined naturally (see the next paragraph). The length $|Y|$ of $Y$
is a random variable, denoted by $L_n$. In this paper we determine
the order of magnitude of the expectation of $L_n$. We show further
that $L_n$ is highly concentrated around its mean, and that the
longest convex chains have a limit shape.
\end{abstract}

\maketitle

\section{Introduction and results}\label{introd}

Let $T \subset \rr$ be a triangle with vertices $p_0,p_1,p_2$ and
let $X \subset T$ be a finite point set. A subset $Y \subset X$ is a
{\sl convex chain} in $T$ (from $p_0$ to $p_2$) if the convex hull
of $Y \cup\{p_0,p_2\}$ is a convex polygon with exactly $|Y|+2$
vertices. A convex chain $Y$ gives rise to the polygonal path $C(Y)$
which is the boundary of this convex polygon minus the edge between
$p_0$ and $p_2$. The {\sl length} of the convex chain $Y$ is just
$|Y|$.

For most part of this paper we assume that $X=X_n$ is a random
sample of $n$ random, uniform, independent points from $T$. Let
$L_n$ be the length of a longest convex chain in $X_n$. The random
variable $L_n$ is a distant relative of the ``longest increasing
subsequence'' problem, cf. \cite{ad}. In this paper we establish
several properties of $L_n$. The first concerns its expectation, $\E
L_n$.

\begin{theorem} \label{limit}
There exists a positive constant $\al$ for which
$$\lim_{n \rightarrow \infty} \frac{ \E L_n}{\sqrt[3]{n}}=\al \,.$$
\end{theorem}

Theorem \ref{limit}, together with some geometric arguments based on
Theorem~\ref{mobius} below, implies that the longest convex chains
have a limit shape $\G$ in the following sense. Let $\cc(X_n)$ be
the collection of all longest convex chains from $X_n$. For every
$\eps>0$
\[
\lim _{n \to \infty} \PP \big(\dist(C(Y),\G)> \eps\mbox{ for some }Y
\in \cc(X_n)\,\big)=0,
\]
where $\dist(.,.)$ stands for the Hausdorff distance. In fact, the
statement of Theorem~\ref{limsh} is much stronger, because there
$\eps$ also converges to $0$. The limit shape turns out to be the
unique parabola arc $\G \subset T$ that is tangent to the sides
$p_0p_1$ at $p_0$ and $p_1p_2$ at $p_2$, see Figure~\ref{parab} a).
The parabola arc $\G$ will be called {\sl the special parabola} in
$T$.

The proof of the 'limit shape' result is based on the following
theorem, saying that $L_n$ is highly concentrated around its
expectation.

\begin{theorem}\label{conc1}
For every $\gamma>0$ there exists a constant $N$, such that for
every $n>N$
$$\PP \big(|L_n -  \E L_n| > \gamma \sqrt{\log n} \; n^{1/6}\big) <
n^{-\gamma^2/14}.  $$
\end{theorem}

For the quantitative version of the limit shape theorem we fix our
triangle $T$ as $T=\conv\{(0,1),(0,0),(1,0)\}$.

\begin{theorem}\label{limsh}
Let $\gamma \geq 1 $ and define $\eps= 3/2 \gamma^{1/2} \, n^{-1/12} (\log
n)^{1/4}$. Then there exists $N>0$, depending on $\gamma$, such
that for every $n>N$,
\[
\PP\big( {\rm dist}(C(Y),\G)>\eps \mbox{ for some }Y \in \cc(X_n)\big)<
2n^{-\gamma^2/14}.
\]
\end{theorem}

\section{Preliminaries}\label{prelim}

When choosing one random point in triangle $T$, the underlying
probability measure is the normalized Lebesgue measure on $T$. Most
of the random variables treated in this paper (e.g. $L_n$) are
defined on the $n$th power of this probability space, to be denoted
by $T ^ {\otimes n}$. In this case $\PP$ denotes the $n$th power of
the normalized Lebesgue measure on $T$.

Throughout the paper, $\A$ stands for the (Lebesgue) area measure on
the plane. So when choosing $n$ independent random points in $T$,
the number of points in any domain $D \subset T$ is a binomial
random variable of distribution $B(n,\A(D)/\A(T))$. Hence the
expected number of points in $D$ is $n \A(D)/\A(T)$.

For binomial random variables we have the following useful deviation estimates,
which are relatives of Chernoff's inequality, see \cite{as}, Theorems
A.1.12 and A.1.13, pp 267-268. If $K$ has binomial distribution with mean value $k>1$
and $c>0$, then
\begin{equation} \label{hoeffding}
\PP \big(K \leq k -c \sqrt{k \log k}\,\big) \leq k^{-c^2/2}.
\end{equation}
On the other hand, for $c>1$,
\begin{equation} \label{hoeffdingfelso}
\PP \big(K \geq ck \big) \leq \left( \frac{e}{c}\right)^{ck}.
\end{equation}
We will use (\ref{hoeffding}) often, mainly with $c=1$.

\begin{figure}[h]
\epsfxsize =\textwidth \centerline{\epsffile{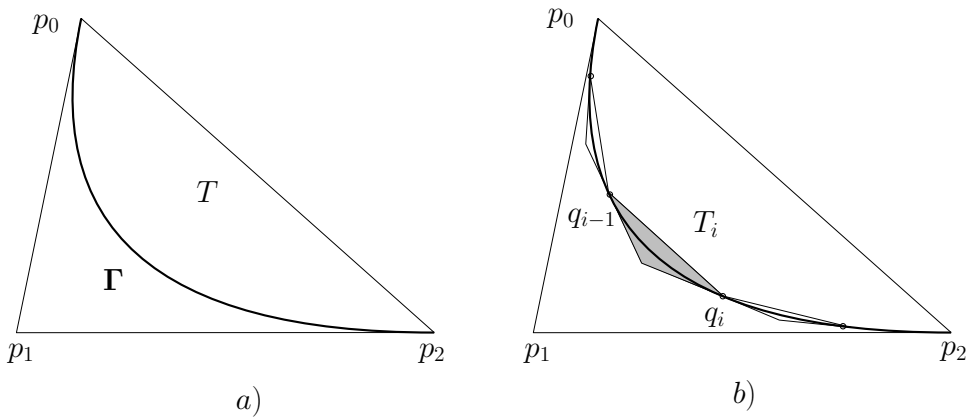}}
  \caption{The special parabola}
  \label{parab}
\end{figure}

The special parabola arc $\G$ in $T$ is characterized by the fact
that it has the largest affine length among all convex curves
connecting $p_0$ and $p_2$ within $T$. (For the definition and
properties of affine arc length see \cite{Bla} or \cite{B99}.) This
is a consequence of the following theorem from \cite{Bla}. Assume
that a line $\ell$ intersects the sides $[p_0,p_1]$ resp.
$[p_1,p_2]$ at points $q_0$ and $q_2$. Let $q_1$ be a point on the
segment $[q_0,q_2]$ and write $T_1$ resp. $T_2$ for the triangle
with vertices $p_0,q_0,q_1$ resp. $q_1,q_2,p_2$, see
Figure~\ref{mobiusfig}.

\begin{theorem}\label{mobius}\cite{Bla} Under the above assumptions
\[
\sqrt[3]{\A(T_1)}+\sqrt[3]{\A(T_2)} \leq \sqrt[3]{\A(T)}.
\]
Equality holds here if and only if $q_1 \in \Gamma$ and $\ell$ is
tangent to $\G$ at $q_1$.
\end{theorem}

The equality part of the theorem implies the following fact. Assume
that $p_0=q_0,q_1,\dots,q_k=p_2$ are points, in this order, on $\G$.
Let $T_i$ be the triangle delimited by the tangents to $\G$ at
$q_{i-1}$ and $q_i$, and by the segment $[q_{i-1},q_i]$,
$i=1,\dots,k$; see Figure~\ref{parab} b).

\begin{corollary}\label{areasum} Under the previous assumptions
$\sum_{i=1}^k\sqrt[3]{\A(T_i)}= \sqrt[3]{\A(T)}$. In particular,
when $\A(T_i)=t$ for each $i=1,\dots,k-1$ and $\A(T_k) < t$, then
$k-1\le \sqrt[3]{\A(T)/t}<k$.
\end{corollary}

We will need a strengthening of Theorem \ref{mobius}. Assume $q_0$
resp. $q_2$ divides the segment $[p_0,p_1]$ resp. $[p_1,p_2]$ in
ratio $a:(1-a)$ and $b:(1-b)$, see Figure~\ref{mobiusfig}.

\begin{figure}[h]
\epsfxsize =0.55\textwidth \centerline{\epsffile{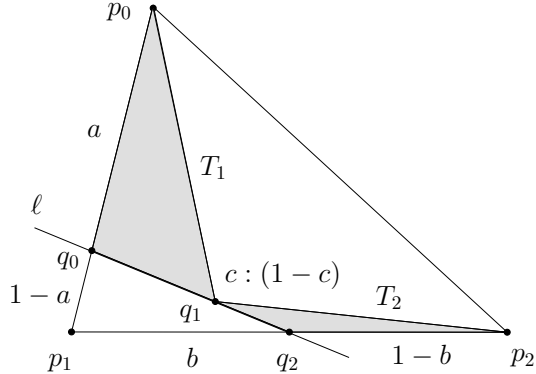}}
  \caption{Characterization of $\Gamma$}
  \label{mobiusfig}
\end{figure}

\begin{theorem}\label{mobi+} With the above notation
\[
 \sqrt[3]{\A(T_1)}+\sqrt[3]{\A(T_2)} \leq \sqrt[3]{\A(T)}-
\sqrt[3]{\A(T)}\frac 13 (a-b)^2.
\]
\end{theorem}

\begin{proof}[{\bf Proof}] Let $c$ be a number between $0$ and $1$
so that $q_1$ divides the segment $[q_0,q_2]$ in ratio $c:(1-c)$.
Then, writing $\A(xyz)$ for the area of the triangle with vertices
$x,y,z$,
$$\A(p_0q_0q_1) = a \A(p_0p_1q_1) = ac \A(p_0p_1q_2) = abc \A(p_0p_1p_2),$$
showing $\A(T_1)=abc\A(T)$. Similarly, $\A(T_2) = (1-a)(1-b)(1-c)
\A(T).$ Hence we have to prove the following fact: $0 \leq a,b,c
\leq 1$ implies
\begin{equation} \label{abc}
1 - \sqrt[3]{abc} - \sqrt[3]{(1-a)(1-b)(1-c)} \geq \frac{1}{3}
(a-b)^2.
\end{equation}
Denote $Q$ the left hand side of (\ref{abc}). By computing the
derivative of $Q$ with respect to $c$ yields that for fixed $a$ and
$b$, $Q$ is minimal when
$$c=\frac{\sqrt{ab}}{\sqrt{ab}+\sqrt{(1-a)(1-b)}} \, .$$
It is easy to see that with this $c$
$$\sqrt[3]{abc} + \sqrt[3]{(1-a)(1-b)(1-c)} =
\left(\sqrt{ab}+\sqrt{(1-a)(1-b)}\right)^{2/3}.$$ Now, denote
$\left(\sqrt{ab}+\sqrt{(1-a)(1-b)}\right)^2$ by $1-u$, so
$$ u = a+b-2ab-2\sqrt{ab(1-a)(1-b)}.$$
We claim that $u \geq (a-b)^2$: this is the same as
$$ a-a^2 + b- b^2 \geq 2\sqrt{(a-a^2)(b-b^2)},$$
which is just the inequality between the arithmetic and geometric
means for the numbers $a-a^2,b-b^2\ge 0$. Therefore, using $u \leq
1$,
\[
Q \geq 1-(1-u)^{1/3} \geq \frac {1}{3}\,u \geq \frac{1}{3}\,(a-b)^2.
\qedhere
\]
\end{proof}

Theorems \ref{mobius} and \ref{mobi+} imply the following

\begin{corollary}\label{erinto}
If $q_1 \in \G$ and $\ell$ is tangent to $\G$ at $q_1$, then with the above
notations, $a=b$.
\end{corollary}

It is clear that the underlying triangle $T$ can be chosen
arbitrarily, since an affine transformation does not influence the
value of $L_n$. Our standard model for $T$ is the one with
$p_0=(0,1)$, $p_1=(0,0)$, $p_2=(1,0)$ as the vertices of $T$. In this case the
special parabola $\G$ has equation $\sqrt{x} + \sqrt{y}=1$.

\section{Other models}\label{models}

There are several choices for the underlying finite set $X$. For
instance, consider the lattice $\frac 1t \z$ where $\z$ is the usual
lattice in $\rr$ and $t>0$ is large, and set $X=T \cap \frac 1t \z$.
Clearly, $n:=|X|\approx \A(T)t^2$ as $t \to \infty$. Write
$Y_n\subset X$ for a longest convex chain in $T$. It is shown in
\cite{BP} that, as $t \to \infty$ (or $n\to \infty$),
\begin{equation}\label{tria}
|Y_n|=\frac 6{(2\pi)^{2/3}}
\sqrt[3]{t^2 \A(T)}(1+o(1))=\frac
6{(2\pi)^{2/3}}n^{1/3}(1+o(1)).
\end{equation}

This result is analogous to Theorem~\ref{limit}, except that in the
lattice case  the value of the constant is known to be
$6/{(2\pi)^{2/3}}$, while in the present paper only the existence of
the limit $\al$ is shown, together with $1.5<\al<3.5$, see
Section~\ref{expect}. This is similar to the longest increasing
subsequence problem, \cite{ad}, where it is easy to see that the
expectation is of order $\sqrt{n}$, but proving the precise
asymptotic formula $2\sqrt n(1+o(1))$ turned out to be difficult,
cf. \cite{ls} and \cite{vk}. In our case, numerical experiments
suggest that $\al=3$ and we venture to conjecture that this is the
actual value of $\al$.

More generally, let $K\subset \rr$ be a convex compact set with
nonempty interior, and set $X_t=K \cap \frac 1t \z$. A set $Y
\subset X_t$ is said to be {\sl in convex position} if no point of
$Y$ lies in the convex hull of the others. In other words, the
convex polygon $\conv \; Y$ has exactly $|Y|$ vertices. Let $Y_t$
be a maximum size subset of $X_t$ which is in convex position and
set $m(K,t)=|Y_t|$. It is shown in \cite{BP} that
\begin{equation}\label{prod}
m(K,t)= \frac 3{(2\pi)^{2/3}}\A^*(K)t^{2/3}(1+o(1))
\end{equation}
where $\A^*(K)$ denotes the supremum (actually, maximum) of the
affine perimeter that a convex subset of $K$ can have. The main
difficulty lies in the case of triangles, that is, proving
(\ref{tria}).

These results can be extended, quite easily, to the present case
when $X_n$ is a random sample of $n$ uniform independent points
from $K$. For instance, writing $Y_n$ for the maximum size subset
of $X_n$  in convex position, one can show the following.
\begin{theorem}\label{positio}Under the above conditions
\[ \lim_{n \rightarrow \infty} n^{-1/3} \E|Y_n|= \frac
{\al \A^*(K)}{2\root 3 \of{A(K)}}.
\]
Here $\al$ is the constant from Theorem~\ref{limit}.
\end{theorem}

One can also prove that $\conv\; Y_n$ has a limit shape, namely,
the unique convex subset of $K$ whose affine perimeter is equal to
$\A^*(K)$. The proofs are almost identical to those used in
\cite{BP}, so we do not repeat them here, instead we rather
explain what is different and more interesting.

Another random model is when $X$ comes from a homogeneous planar
Poisson process $X(n)$ of intensity $n/ \A(T)$. Given a domain $D$
in the plane, $m(D)=|X(n) \cap D|$, the number of points in $D$,
has Poisson distribution with parameter $\lambda=n \A(D)/\A(T)$,
i.e.
$$ \PP\big(m(D)=k\big)= e^{-\lambda} \lambda ^k / k! \;.$$

We can also think of the Poisson model as follows: for a domain $D$,
we first pick a random number $m$ according to the corresponding
Poisson distribution, and then choose $m$ random, independent,
uniform points in $D$. The advantage of the Poisson model is that
the number of points of $X(n)$ in disjoint domains are independent
random variables, unlike in the uniform model.

As is well known, the uniform model $X_n$ and the Poisson model
$X(n)$ behave very similarly. In particular, Theorems \ref{limit},
\ref{conc1}, and \ref{limsh} remain valid for the Poisson model as
well, with essentially the same quantitative estimates. The proofs
are quite standard, and we do not go into the details. Actually,
the proof of Theorem \ref{limsh} is simpler in the Poisson model
since there the subtriangles behave the same way as any other
triangle.

The longest increasing subsequence problem has been almost
completely solved by now, see \cite{ad}. In this respect, our
results only constitute the first, and perhaps the simplest, steps
in understanding the random variable $L_n$.

\section{Expectation}\label{expect}

The main target of this section is to prove of
Theorem~\ref{limit}. We also establish upper and lower bounds for
the constant involved.

\begin{proof}[{\bf Proof} {\rm of Theorem~\ref{limit}}]
We start with an upper bound on $\E L_n$:
\begin{equation}\label{felso}
\limsup_{n \rightarrow \infty} \frac{\E L_n}{\sqrt[3]{n}} \leq
\sqrt[3]{2}e = 3.4248\dots.
\end{equation}
It is shown in \cite{B99}, equation (5.3) (cf. \cite{BRSZ} as well)
that the probability of $k$ uniform independent random points in $T$
forming a convex chain is
\[
\frac{2^k}{k! \,(k+1)!} \;.
\]
Therefore the probability that a convex chain of length $k$ exists
is at most $ {n \choose k} {2^k}/(k! \,(k+1)!)$. In other words
\begin{equation*}\label{upp}
\PP(L_n\ge k) \le {n \choose k} \frac{2^k}{k! \,(k+1)!}\;.
\end{equation*}
We use this estimate and Stirling's formula to bound $\E L_n$.
Assume $\gamma > \sqrt[3]{2}e$. Then
\begin{align*}
\E L_n &= \sum_{k=0}^{n} \PP(L_n>k) \leq \sum_{k=0}^{n} \PP(L_n \ge k)\\
&\leq \gamma \sqrt[3]{n} + \sum_{k > \gamma \sqrt[3]{n}}
\PP(L_n\ge k) \\
&\leq  \gamma \sqrt[3]{n} + \sum_{k > \gamma \sqrt[3]{n}} {n \choose
k} \frac{2^k}{k! \,(k+1)!}\\
&\leq \gamma \sqrt[3]{n} + \sum_{k > \gamma \sqrt[3]{n}} \frac{(2n)^k}{(k!)^3} \\
&\leq \gamma \sqrt[3]{n} + \sum_{k > \gamma \sqrt[3]{n}}
\frac{1}{\sqrt{(2\pi \gamma)^3 n}} \left(\frac{2\,e^3}{\gamma^3}\right)^k\\
&\leq \gamma \sqrt[3]{n} + n^{-1/2} C,
\end{align*}
where $C = \gamma^3 / (\gamma^3 - 2 e^3)$ is a positive constant.
Since this holds for arbitrary $\gamma > \sqrt[3]{2}\,e$,
(\ref{felso}) is proved.

Next we establish a lower bound for $\E L_n$. We use the second half
of Corollary~\ref{areasum} with $t=2\A(T)/n$. So we have triangles
$T_i$ of area $t$ for $1 \leq i \leq k-1$, and the last
\label{intro}triangle $T_k$ of area less than $t$. By
(\ref{areasum}) $ k \geq \sqrt[3]{n/2}$. Let $X_n$ be the uniform
independent sample from $T$. Let $x_i$ be a point of $T_i \cap X_n$,
provided that $T_i \cap X_n \ne \emptyset$. The collection of such
$x_i$'s forms a convex chain. Hence the expected length of the
longest convex chain is at least the expected number of non-empty
triangles $T_i$, so
\begin{align*}
\E L_n \; &\geq \; \sum_1^k \PP\big(T_i\cap X_n \ne \emptyset\big)
\; \geq \;
(k-1)\left(1-\left(1-\frac{2}{n}\right)^n\right) \\
&\ge \left(\sqrt[3]{\frac{n}{2}}-1\right) \, \left(1-e^{-2}\right)
\approx 0.6862 \, n^{1/3}.
\end{align*}

\noindent What we have proved so far is that
\[
\underline{\al} = \liminf_{n \rightarrow \infty}n^{-1/3}\E L_n >
0.6862,\mbox{ and } \overline{\al}= \limsup_{n \rightarrow \infty}
n^{-1/3}\E L_n < 3.4249.
\]
We show next that the limit exists. Suppose on the contrary that
$\underline{\al}< \overline{\al}$.

The idea of the proof is to use the second half of
Corollary~\ref{areasum} again, with the longest convex chain in the
small triangles having length close to the limsup , while in the
large triangle, $\E L_n$ is close to the liminf. For convenience, we suppose that
$A(T)=1.$

Choose a large $n$ with $ \E L_n\geq (1-\eps)\,\overline{\al}
\sqrt[3]{n}$, and an $N$ much larger than $n$ with $ \E L_N\le
(1+\eps)\, \underline{\al} \sqrt[3]{N}$. Here $\eps$ is a suitably
small positive number. Define $n_1$ so that the equation $n = n_1 -
\sqrt{n_1 \log n_1}$ holds.

Choose $N$ uniform, independent random points from triangle $T$.
Define $t = n_1 / N$. Hence the expected number of points in a
triangle (contained in $T$) of area $t$ is $n_1$.

Apply the second half of Corollary \ref{areasum} with this $t$. Then
the number of triangles, $k$, satisfies  $k > \sqrt[3]{N/ n_1}$.

Denote by $k_i$ the number of points in $T_i$, and by $\E L^i$ the
expectation of the length of the longest convex chain in $T_i$.
Clearly $k_i$ has binomial distribution with mean $n_1$, except for
the last triangle where the mean is less than $n_1$.

Since the union of convex chains in the triangles $T_i$ is a convex
chain in $T$ between $(0,0)$ and $(1,1)$, by estimate
(\ref{hoeffding}) we have
\begin{align*}
 \E L_N &\geq \sum_{i \leq k}  \E L^i \; \ge \;
 \sum_{i\leq k-1} \PP (k_i > n)  \E L_{n}\\
&\geq \sum_{i\leq k-1} \left(1- n_1^{-1/2}\right) (1-\eps) \,
\overline{\al}
\sqrt[3]{n}\\
&\geq \left(\sqrt[3]{N/ n_1}-1\right)\left(1- n_1^{-1/2}\right)
(1-\eps) \,\overline{\al} \sqrt[3]{n}\\
&=\overline{\al}\,
\sqrt[3]{N}(1-\eps)\left(1-n_1^{-1/2}\right)\left(\sqrt[3]{n/n_1}-\sqrt[3]{n/N}\right)\\
&\ge \overline{\al} \,\sqrt[3]{N}(1-2\eps),
\end{align*}
where the last inequality holds if $n$ is chosen large enough and
$N$ is chosen even larger with $n/N$ very small. Thus $(1+\eps)\,
\underline{\al} \ge (1-2\eps)\, \overline{\al}$ which, for small
enough $\eps$, contradicts our assumption $\underline{\al} <
\overline{\al}$.
\end{proof}

\noindent{\bf Remark.} The lower bound $\E L_n \ge 0.6862 \,
n^{1/3}$ is probably the easiest to prove. A better estimate, also
mentioned by Enriquez \cite{En}, can be established as follows. Assume $T$ is
the standard triangle and let $D$ denote the domain of $T$ lying
above $\Gamma$. Then $\A(D) = 1/3$, so the expected number of points
in $D$ is $2n/3$, and the number of points is concentrated around
this expectation. The affine perimeter of $D$ is $2 \sqrt[3]{1/2}$
(see \cite{B99}), and a classical result of R\'enyi and Sulanke
\cite{RS} yields that expected number of vertices of $\conv(D \cap
X_n)$ is about
$$ \Gamma\left(\frac{5}{3}\right) \sqrt[3]{\frac{2}{3}}
\left(\frac{1}{3} \right)^{-1/3} 2\sqrt[3]{1/2} \, \sqrt[3]{2n/3}
\approx 1.5772 \,\sqrt[3]{n}$$
Since most vertices are located next
to the parabola, the majority of them form a convex chain, and so
\begin{equation}\label{also}
\liminf_{n \rightarrow \infty} \frac{ \E L_n}{\sqrt[3]{n}} \geq
1.5772\dots.
\end{equation}
This sketch can be completed with standard tools. From now on, we will use this
estimate. Also, $\alpha$  will always refer to the limit constant of
Theorem~\ref{limit}.

\section{Concentration results for $\E L_n$}\label{concentr}

The concentration results proved here are consequences of
Talagrand's inequality from \cite{tal} which says the following.
Suppose $Y$ is a real-valued random variable on a product
probability space $\Omega^{\otimes n}$, and that $Y$ is 1-Lipschitz
with respect to the Hamming distance, meaning that
$$ | Y(x)-Y(y)| \leq 1$$
whenever $x$ and $y$ differ in one coordinates. Moreover assume that
$Y$ is {\sl $f$-certifiable}. This means that there exists a
function $f: \mathbb{N} \rightarrow \mathbb{N}$ with the following
property: for every $x$ and $b$ with $Y(x) \geq b$ there exists an
index set $I$ of at most $f(b)$ elements, such that $Y(y) \geq b$
holds for every $y$ agreeing with $x$ on $I$. Let $m$ denote the
median of $Y$. Then for every $s>0$ we have
\\
$$ \PP (Y \leq m-s)\leq 2 \, \mathrm{exp}\left(\frac{-s^2}{4f(m) } \right)$$
and
$$ \PP (Y \geq m+s)\leq 2 \, \mathrm{exp}\left(\frac{-s^2}{4f(m+s) } \right).$$

\noindent When applied to $L_n$, these inequalities prove
concentration about the median, to be denoted by $m_n$.
Theorem~\ref{conc1} concerns the mean of $L_n$. However, concentration
ensures that the mean and the median are not far apart, in fact,
$\lim n^{-1/3}m_n = \al$. First we need a lower bound on $m_n$.

\begin{lemma}\label{median}  Suppose that $\log n > 25$.
Then $$m_n \ge \sqrt[3]{3n/\log n}.$$
\end{lemma}

\noindent Since this is a special case of Lemma \ref{mediansub} from
the next section, the proof will be given there.

\begin{proof}[{\bf Proof} {\rm of Theorem~\ref{conc1}}] The statement cries out for
the application of Talagrand's inequality. The random variable $L_n$
satisfies the conditions with $f(b)=b$, since fixing the coordinates
of a maximal chain guarantees that the length will not decrease, and
changing one coordinate changes the length of the maximal chain by
at most one. Write $m=m_n$ for the median in the present proof.
Setting $s=\be\sqrt{m \log m}$ where $\be$ is an arbitrary positive
constant, we have
\begin{align*}
\PP \big(|L_n-m| \geq \be \sqrt{m \log m}\,\big) &< 4
\exp\left\{\frac
{-\be^2 m\log m}{4(m+\be \sqrt{m\log m})} \right\}\\
&= 4 \exp\left\{\frac{-\be^2\log m}{4(1+\be \sqrt{m^{-1}\log m})}
\right\}
\end{align*}

\noindent Define now $\be_0= c \sqrt{m / \log m}$ with a constant $c>0$, which will
be fixed at the end of the proof in order to give the correct estimate. If $\be
\leq \be_0$, then $\be\sqrt{m^{-1}\log m} \le c$, and the
denominator in the exponent is at most $4(1+c)$. Thus
\begin{equation}\label{betasmall}
\PP \big(|L_n-m| \geq  \be \sqrt{m \log m}\, \big) <
4m^\frac{-\be^2}{4(1+c)}.
\end{equation}
On the other hand, for $\be > \be_0$ we have
\begin{align}\label{betabig}
\PP\big(|L_n-m| \geq  \be \sqrt{m \log m}\, \big) &< \PP
\big(|L_n-m| \geq
\be_0\sqrt{m\log  m}\,\big)\\
\nonumber &= 4\exp \left(-m \, \frac{c^2}{4(1+c)}\right).
\end{align}
Next, we compare the median and the expectation of $L_n$.
\begin{equation*}
| \E L_n - m | \leq   \E |L_n -m |= \int_{0}^{\infty} \;\PP(|L_n-m |
> x) dx.
\end{equation*}
The range of $L_n$ is $[1,n]$, so the integrand is $0$ if $x>n$.
Substitute $x=\be\sqrt{m\log m}$, and divide the integral into two
parts at $\be_0$:
\[
|\E L_n-m| \leq 4\sqrt{m \log m}(I_1+I_2),
\]
where
\begin{equation}\label{I1}
I_1=\int_0^{\be_0}m^{-\be^2/4(1+c)}d\be<\int_0^{\infty}m^{-\be^2/4(1+c)}d\be=
\sqrt{\frac {\pi (1+c)}{\log m}},
\end{equation}
and
\begin{equation}\label{I2}
I_2=\int_{\be_0}^{n/\sqrt{m\log m}} \exp \left(-m \, \frac{c^2}{4(1+c)}\right) d\be
< n \exp \left(-m \, \frac{c^2}{4(1+c)}\right).
\end{equation}
By Lemma \ref{median}, $n < m^4$, so $I_2< m^4\exp (-m \, c^2/4(1+c))$. Since
$m_n$ goes to infinity as $n$ increases (again by Lemma
\ref{median}), the bound on $I_2$ is eventually much smaller than the one on $I_1$:
\begin{align}\label{exp-med}
|\E L_n-m| &\leq 4\sqrt{m\log m}(I_1+I_2) \nonumber \\
&<4\sqrt{\pi (1+c) m}+4\sqrt{m\log m}\,m^4 \exp  \left(-m \,
\frac{c^2}{4(1+c)}\right) \\
&\le 5 \sqrt{\pi (1+c)} \sqrt m \nonumber
\end{align}
for all large enough $n$. Hence $\E L_n$ is of the same order of
magnitude as $m_n$, and we obtain
\begin{equation}\label{medianlim}
\lim n^{-1/3}\E L_n= \lim n^{-1/3} m_n =\al.
\end{equation}
For fixed $\gamma$ and for large enough $n$, (\ref{exp-med}) implies
\begin{align*}
&\PP \big(|L_n -  \E L_n| > \gamma \sqrt{\log n} \;
n^{1/6}\big)\\
&\leq \PP \big(|L_n -  m| > \gamma \sqrt{\log n} \;
n^{1/6}-|\E L_n-m|\, \big)\\
&\leq \PP \big(|L_n -  m| > \gamma \sqrt{\log n} \;
n^{1/6}- 5 \sqrt{\pi (1+c)}\sqrt{m}\, \big).
\end{align*}
Using $m_n \le 3.43 n^{1/3}$ from (\ref{felso}) and
(\ref{medianlim}), it is easy to see that
\begin{align*}
\gamma \sqrt{\log n} \;n^{1/6}- 5 \sqrt{\pi (1+c)\, m} &\ge \g \sqrt{m} \left(
 \sqrt{\frac{3 \log m- \log 41}{3.43}}-\frac { 5 \sqrt{\pi (1+c)}}{\g} \right)\\
&\ge \g \sqrt {\frac {3 }{3.44}} \sqrt {m \log m}.
\end{align*}
Since for large enough $n$, $\g \sqrt {3/3.44} < \beta_0= c
\sqrt{m / \log m}$, (\ref{betasmall}) finally implies
\begin{align*}
&\PP \big(|L_n -  \E L_n| \geq \gamma \sqrt{\log n} \;n^{1/6}\big)\\
&\leq \PP \big(|L_n - m|\ge \g \sqrt {\frac{3} {3.44}}\sqrt{m \log m }\,\big)\\
&\leq 4 m^{- 3 \gamma^2  / 13.76 (1+c) }  \leq n^{-\gamma^2 / 14}
\end{align*}
with (\ref{medianlim}) and the choice of $c=0.01$.
\end{proof}

\noindent{\bf Remark.} The constant in the exponent is far from
being best possible. We have made no attempt to find its optimal
value. In general, Talagrand's inequality is too general to give the
precise concentration, see Talagrand's comments on this in
\cite{tal}.

\section{Subtriangles}\label{subtri}

For the proof of Theorem~\ref{limsh} we need to consider
subtriangles $S$ of $T$, that is, triangles of the form
$S=\conv\;\{a,b,c\}$ with $a,b,c \in T$, while $X_n$ is still a
random sample from $T$. We will need to estimate the concentration
of the longest convex chain from $X_n$ in $S$. Since this random
variable depends only on the relative area of $S$, we may and do
assume that $T$ is the standard triangle and $S=\conv\{(0,\sqrt
{s}),(0,0),(\sqrt {s},0)\}$. Thus $\A(S)=s/2$. Write $L_{s,n}$ for
the length of the longest convex chain in $S$ from $(0,\sqrt {s})$ to
$(\sqrt {s},0)$, and $m_{s,n}$ for its median. In the
following statements, we consider the situation when $sn/2$, the
expected number of points from $X_n$ in $S$, tends to infinity.

As in the proof of Theorem~\ref{conc1}, we need two estimates: a
lower bound for the median guarantees that the mean and the median
are close to each other, while an upper bound for the expectation (or for the
median) is needed for deriving the inequality in terms of $n$. Here
comes the lower bound; the case $s=1$ is Lemma \ref{median}.

\begin{lemma}\label{mediansub} Suppose that $\log (ns) > 25$.
Then $$m_{s,n} \ge \sqrt[3]{3ns/\log (ns)}.$$
\end{lemma}

\begin{proof}[{\bf Proof}] Set $t=(\A(S) \log (ns))/(3ns)$, and apply the
second half of Corollary~\ref{areasum} to the triangle $S$. The
number of triangles is $k$ with
\[
\sqrt[3]{3ns/\log (ns)}<k \leq \sqrt[3]{3ns/\log (ns)} +1.
\]
For any $i \in \{1,\dots,k$\}, the probability that $T_i$ contains
no point of $X_n$  is
\begin{align*}
\PP(T_i\cap X_n=\emptyset)&\leq\left(1-\frac {\log
(ns)}{3ns}\right)^n
\\ &< \exp\left(\frac{-\log(ns)}{3s}\right)= (ns)^{-1/3s}<
(ns)^{-1/3}.
\end{align*}
Hence the union bound yields
\begin{align*}
\PP\big(L_{n,s} > \sqrt[3]{3ns/\log (ns)}\,\big) &\ge  1-\PP(T_i\cap
X_n =
\emptyset \mbox{ for some }i \le k) \\
&\ge 1-k\,(ns)^{-1/3}\\
&\ge 1- \big( \sqrt[3]{3/\log (ns)} + (ns)^{-1/3} \big),
\end{align*}
which is greater than $1/2$ by the assumption.
\end{proof}

Obtaining an upper bound for the mean is slightly more delicate; note that in the
Lemma below $s$ need not be fixed.
\begin{lemma}\label{meansub} Assume $ns \to \infty$. Then
$$\lim \, (ns)^{-1/3}\E L_{s,n} =\alpha$$
where $\alpha$ is the same constant as in Theorem~\ref{limit}.
\end{lemma}

\begin{proof}[{\bf Proof}] Take any $\eps >0$ and choose $N_0$ (depending on $\eps$)
so large that for every $k \geq N_0$,
$(1-\eps) \alpha < \E L_k \, k^{-1/3} < (1+\eps) \alpha$. The random variable
$K=|X_n\cap S|$
has binomial distribution with mean $ns$. When $ns$ is large enough, $ns - \sqrt{ns
\log ns } \geq N_0$, and we use (\ref{hoeffding}) for a lower estimate:
\begin{align*}
\E L_{s,n} &= \sum_{k=0}^{n} \PP (K=k) \E L_k\\
&\geq \PP(K> ns - \sqrt{ns \log ns}) (1-\eps)\, \alpha \, (ns - \sqrt{ns \log ns})
^{1/3}\\
&\geq (1- (ns) ^{-1/2})(1-\eps) \, \alpha \, (ns - \sqrt{ns \log ns}) ^{1/3}\\
&\geq (1-2 \eps)\,  \alpha \, (ns)^{1/3}.
\end{align*}
For the upper bound, Jensen's inequality applied to
$\sqrt[3]{x}$ comes in handy:
\begin{align*}
\E L_{s,n} &= \sum_{k=0}^{n} \PP (K=k) \E L_k\\
&\leq N_0 \, \PP(K<N_0)
+ \sum_{k=N_0}^{n} \PP (K=k) \E L_k
\\
&\leq N_0 \, + \sum_{k=N_0}^{n} \PP (K=k) \, (1+\eps) \, \alpha \, \sqrt[3]{k}
\\
&\leq N_0 \, + \PP(K \geq N_0) \, (1+\eps) \, \alpha \,\left(
\sum_{k=N_0}^{n} \frac{\PP (K=k)}{\PP(K \geq N_0)} \ k \right)^{1/3}\\
&\leq N_0 \, + \PP(K \geq N_0)^{2/3} \, (1+\eps) \, \alpha \,(
 \E \, K)^{1/3}\\
&\leq N_0 + (1+\eps) \, \alpha \, (ns)^{1/3}\leq (1+2 \eps) \, \alpha \, (ns)^{1/3}.
\qedhere
\end{align*}
\end{proof}

Next, we derive the strong concentration property of $L_{s,n}$, the
analogue of Theorem~\ref{conc1}.

\begin{theorem}\label{concsub}
Suppose $\tau$ is a constant with $0 \leq \tau < 1$. Then for every $\gamma>0$ there
exists a constant
$N$, such that for every $n>N$ and every $s\geq
n^{-\tau}$,
$$\PP \big(|L_{s,n} -  \E L_{s,n}| > \gamma \sqrt{\log ns} \; (ns)^{1/6}\big) <
(ns)^{-\gamma^2/14}.  $$
\end{theorem}

\begin{proof}[{\bf Proof}] This proof is almost identical with that of Theorem
\ref{conc1}. Since $L_{s,n}$ is a random variable on $T ^ {\otimes
n}$, we can apply Talagrand's inequality with the certificate
function $f(b)=b$ in the same way as in the proof of Theorem~\ref{conc1}. Write
again $m$ for $m_{s,n}$, the median of $L_{s,n}$. Define $\be_0=
c \sqrt{m / \log m}$ with $c=0.01$, then the estimates (\ref{betasmall})
and (\ref{betabig}) remain valid with $L_{s,n}$ in place of $L_n$.
Just as before,
\begin{align*}\label{submedianmean}
| \E L_{s,n} - m| &\leq   \E |L_{s,n} - m |=
\int_{0}^{\infty} \;\PP(|L_{s,n}-m |> x) dx \\
&= 4\sqrt{m \log m}(I_1 + I_2)
\end{align*}
where $I_1$ and $I_2$ are defined the same way as in (\ref{I1})
and (\ref{I2}). Moreover, $I_1$ satisfies the inequality
(\ref{I1}). With $I_2$ we have to be a bit more careful.

Note that $s \geq n^{-\tau}$ with $\tau <1$ guarantees that
Lemma~\ref{mediansub} is applicable for $n > \exp(25/(1-\tau))$. As $x/
\log x$ is monotone increasing for $x > e$,
\[
m \ge \sqrt[3]{\frac{3ns}{\log (ns)}} \ge
\sqrt[3]{\frac{3n^{1-\tau}}{(1-\tau) \log n}} >
\sqrt[3]{\frac{n^{1-\tau}}{ n^{(1-\tau)/2}}} = n^{(1-\tau)/6}
\]
for large enough $n$, and therefore by (\ref{I2})
\[
I_2 < m^{6/(1-\tau)} \exp \left(-m \, \frac{c^2}{4(1+c)}\right)
\]
where of course $6/(1-\tau) <  \infty$. Lemma~\ref{mediansub}
implies that $m = m_{s,n} \rightarrow \infty$, thus  the bound on $I_2$ is much
smaller than the one on $I_1$ for large enough $n$. Therefore, just as in
(\ref{exp-med}),
\begin{align*}
|\E L_{s,n}-m| &\leq 4\sqrt{m\log m}(I_1+I_2)\\
&<4\sqrt{\pi (1+c) m}+4\sqrt{m\log m}\,m^{6/(1-\tau)} \exp  \left(-m \,
\frac{c^2}{4(1+c)}\right) \\
&\le 5 \sqrt{\pi (1+c)} \sqrt m.
\end{align*}
Hence $\E L_{s,n}$ is of the same order of magnitude as $m=m_{s,n}$. Since
$sn \geq n^{1-\tau} \rightarrow \infty$, we can use
Lemma~\ref{meansub}, obtaining that for large enough $n$,
\begin{equation}\label{submedfelso}
m_{s,n} \le 3.431 \sqrt[3]{ns}.
\end{equation}
Again for fixed $\gamma$ and for large enough $n$,
\begin{align*}
&\PP (|L_{s,n} -  \E L_{s,n}| > \gamma \sqrt{\log ns} \;
(ns)^{1/6})\\
&\leq \PP (|L_{s,n} -  m| > \gamma \sqrt{\log ns} \;
(ns)^{1/6}-|\E L_{s,n}-m|)\\
&\leq \PP (|L_{s,n} -  m| > \gamma \sqrt{\log ns} \;
(ns)^{1/6}-5 \sqrt{\pi (1+c)} \sqrt m),
\end{align*}
and by (\ref{submedfelso}),
\begin{align*}
\gamma \sqrt{\log ns} \;
(ns)^{1/6}-5 \sqrt{\pi (1+c)} \sqrt m &\ge \g \sqrt {\frac {3 }{3.44}} \sqrt {m \log
m}.
\end{align*}
Since for large enough $n$, $\g \sqrt {3/3.44} < \beta_0=c
\sqrt{m / \log m}$, (\ref{betasmall}) applied to $L_{s,n}$ and
(\ref{submedfelso}) finally implies
\begin{align*}
&\PP \big(|L_{s,n} -  \E L_{s,n}| \geq \gamma \sqrt{\log ns} \;(ns)^{1/6}\big)\\
&\leq \PP \big(|L_{s,n} - m|\ge \g \sqrt {\frac{3} {3.44}}\sqrt{m \log m }\,\big)\\
&\leq 4 m^{- 3 \gamma^2  / 13.76 (1+c) }  \leq (ns)^{-\gamma^2 /
14}.\qedhere
\end{align*}
\end{proof}

\noindent {\bf Remark.} The proof also yields that for any
$0<A<B<\infty$, there exists $N$ (depending on $A$ and $B$ only),
such that the inequality of Theorem~\ref{concsub} holds for any
$\gamma \in [A,B]$ and for every $n>N$.

\section{Geometric lemmas}\label{geometry}

For the proof of Theorem~\ref{limsh} we need further preparations.
We start by assuming that $K$ is a convex compact set in the plane
and $A(K)>0$, and $\widetilde{X}_n$ is a random sample of $n$ uniform and
independent points from $K$. We need to estimate the probability
that $\widetilde{X}_n$ is in convex position, that is, no point of $\widetilde{X}_n$ is
contained in the convex hull of the others. We denote this
probability by $\PP(\widetilde{X}_n \mbox{ convex in }K)$.

\begin{lemma}\label{l:convpos}If $K$ is as above,
\[
\PP(\widetilde{X}_n \mbox{ {\rm convex in} }K) < \left(\frac
{240}{n^2}\right)^n.
\]
\end{lemma}

\begin{proof}[\bf{Proof.}] Let $P$ be the smallest area
parallelogram containing $K$. As is well known, $A(P) \le 2A(K)$.
Let $X_n^*$ be a random sample of $n$ uniform and independent
points from $P$. In this case a (surprisingly exact) result of
Valtr~\cite{val} says that $$\PP(X_n^* \mbox{ convex in }P)=
(n!)^{-2}{{2n-2}\choose {n-1}}^2.$$ Now we have
\begin{align*}
\PP(\widetilde{X}_n \mbox{ convex in }K)&=\PP(X_n^* \mbox{ convex in }P\,| \, X_n^*
\subset K)\\
&= \frac {\PP(X_n^* \mbox{ convex in }P\mbox{ and }X_n^*
\subset K)}{\PP(X_n^* \subset K)}\\
&\le  \frac {\PP(X_n^* \mbox{ convex in }P)}{\PP(X_n^* \subset K)}\\
&= (n!)^{-2}{{2n-2}\choose {n-1}}^2\left(\frac
{A(P)}{A(K)}\right)^n < \left(\frac {240}{n^2}\right)^n,
\end{align*}
where the last step is a straightforward estimate.
\end{proof}

>From now on we work exclusively with the standard triangle $T$.

Assume next that $K$ is a convex subset of the triangle $T$, and
let $X_n$ be random sample of $n$ uniform and independent points
from $T$. We define $M(K,n)$ as the random variable
\[
M(K,n)=\max \{|Y|: Y\subset X_n \cap K\mbox{ is in convex
position}\}.
\]
>From Theorem~\ref{positio} it is not hard to determine what the
asymptotic expectation of $M(K,n)$ is. But what we need is that
$M(K,n)$ is large with small probability. This is the content of
the next lemma.

\begin{lemma}\label{l:M>m} Let $K$ be a convex subset of $T$.
Then for any positive integers $n$ and $\mu$
satisfying  $1920\, e^2 \, A(K)\, n\le \mu ^3$,
\[
\PP(M(K,n)\ge \mu)\le \mu ^3 2^{-\mu }+ n2^{-\mu^3/(480e)}.
\]
\end{lemma}

\begin{proof}[\bf{Proof.}] If $M(K,n)\ge \mu$, then $K \cap X_n$
contains a subset of size $\mu$ which is in convex position.
Lemma~\ref{l:convpos} and the union bound imply that
\[
\PP\big(M(K,n)\ge  \mu \big|  |K\cap X_n|=k\big) \le {k \choose
\mu }\left(\frac {240}{\mu^2}\right)^{\mu} \le \left(\frac {240\,
e \,k}{\mu ^3}\right)^{\mu }.
\]
The random variable $|K\cap X_n|$ has binomial distribution. Thus
we have
\begin{align*}
&\PP(M(K,n)\ge \mu)\\
&=\sum_{k=\mu}^n\PP \big( M(K,n)\ge \mu \big|  |K\cap X_n|=k
\big){n
\choose k}(2A(K))^k(1-2A(K))^{n-k}\\
&\le \sum_{k=\mu }^n \min \left\{1,\left(\frac
{240\,e\,k}{\mu^3}\right)^{\mu}\right\}{n
\choose k}(2A(K))^k(1-2A(K))^{n-k}\\
&=\sum_{k < k_0}[..]+\sum_{k=k_0}^{n}[..].
\end{align*}
Here we choose $k_0$ to be equal to $\mu^3/(480e)$. Then
\[
\sum_{k < k_0}[..]\le \sum_{k < k_0}\left(\frac {240\,e\,k_0}{\mu
^3}\right)^{\mu} <k_0 2^{-\mu}< \mu ^3 2^{-\mu}.
\]
Since ${n \choose k}(2A(K))^k(1-2A(K))^{n-k}$ is decreasing for $k
>2A(K)n$, and the condition on $\mu$ guarantees that $k_0 >2A(K)n$,
\begin{eqnarray*} \sum_{k>k_0}[..] &\le& n{n \choose
k_0}(2A(K))^{k_0}(1-2A(K))^{n-k_0}\\
&\le& n\left(\frac{ne}{k_0}\right)^{k_0}(2A(K))^{k_0} = n
\left(\frac{2eA(K)n}{k_0}\right)^{k_0}\\
&<&n2^{-k_0}=n2^{-\mu ^3/(480e)}.
\end{eqnarray*}
\end{proof}

For the proof of Theorem~\ref{limsh} we will consider other
parabolas that are similar to $\G$. Let $\G_r$ be the parabola
defined by the equation $\sqrt x + \sqrt y=\sqrt{1+r}$ where the
parameter $r \in (-1,3)$. The graph of $\G_r$ is the homothetic
copy of $\G$ with ratio of homothety $1+r$, and center of
homothety at the origin, see Figure~\ref{parab2}~a). Assume the
point $(a,b)$ is on $\G$. Then the point $((1+r)a,(1+r)b)$ is on
$\G_r$, and the tangent line to this point on $\G_r$ is given by
the equation
\[
\frac x{\sqrt a}+\frac y{\sqrt b}=1+r.
\]
It follows that the distance between parallel tangent lines to
$\G$ and $\G_r$ is
\begin{equation}\label{eq:dist}
\frac {|r|}{\sqrt{\frac 1a+\frac 1b}} \le \frac {|r|}{\sqrt 8}.
\end{equation}
Define now
\[
\rho=\sqrt 8 \eps=3 \sqrt{2} \g^{1/2}n^{-1/12}(\log n)^{1/4},
\]
here $\eps$ comes from Theorem~\ref{limsh}. This definition
immediately implies the following fact.

\begin{prop}\label{pr:dist}If a convex chain $C(Y)$ lies between
$\G_{-\rho}$ and $\G_\rho$, then $\dist(C(Y),\G)\le \eps$.
\end{prop}

We need one more piece of preparation. Assume $\ell$ is a tangent
to $\G_r$, at the point $q$. With the notations of
Section~\ref{prelim}, let $T_1$ and $T_2$ denote the two triangles
determined by $\ell$ and $q$, see Figure~\ref{parab2} a). Let
$X_n$ be a random sample of $n$ points from $T$ and let $L^i$
denote the length of the longest convex chain in $T_i$, $i=1,2$.

\begin{figure}[h]
\epsfxsize =\textwidth \centerline{\epsffile{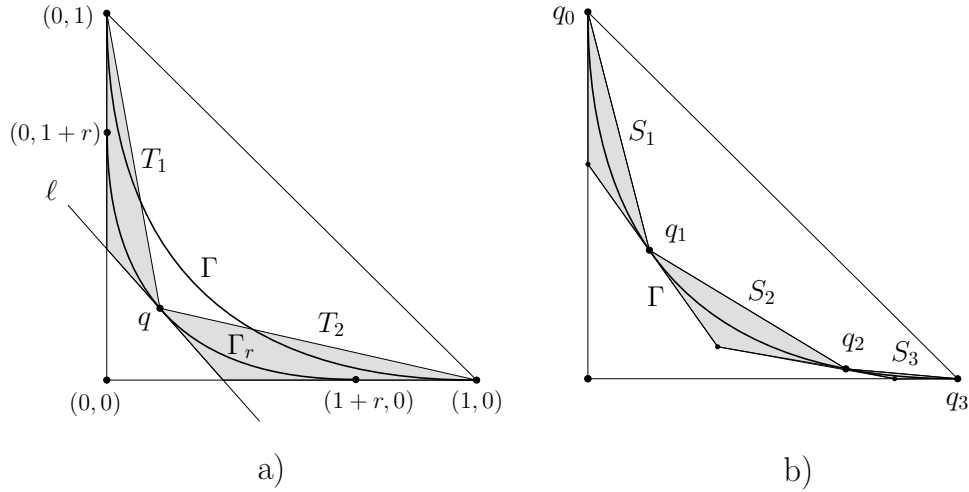}}
  \caption{Convex chains far from $\Gamma$}
  \label{parab2}
\end{figure}

\begin{lemma}\label{l:2chains} For sufficiently large $n$, if $|r|\ge n^{-1/12}$, then
\[
\E L^1+\E L^2 \le \E L_n -0.52 \,r^2\sqrt[3]{n}.
\]
\end{lemma}

\begin{proof}[\bf{Proof.}] Let $t_i = 2 \A(T_i)$ for $i=1,2$. We want to apply
Theorem~\ref{mobi+}. It is not hard to see (using
Corollary~\ref{erinto} for instance) that what is denoted by
$|a-b|$ there, is equal to $|r|$ here. Consequently
\begin{equation}\label{terhiany}
 \sqrt[3]{t_1 /2} + \sqrt[3]{t_2 /2} \leq  \sqrt[3]{1/2} -\sqrt[3]{1/2} \; \frac 1 3
\, r^2.
\end{equation}

Write $L^i$ for the longest convex chain in the triangle $T_i$. By
affine invariance $L^i$ has the same distribution as $L_{t_i,n}$
(from Section~\ref{subtri}) for $i=1,2$. We need to estimate $\E
L_n- (\E L^1 + \E L^2)$ from below.

For four points $q_0=(0,1)$, $q_1$, $q_2$ and $q_3=(1,0)$ in this
order on  $\G$, denote by $S_i$ the triangle delimited by the
tangents to $\G$ at $q_{i-1},q_i$, and by the segment
$[q_{i-1},q_i]$, $i=1,2,3$; see Figure~\ref{parab2} b). Choose
$q_1$ and $q_2$ so that $\A(S_1)=t_1/2$ and $\A(S_2)=t_2/2$. Then
Corollary~\ref{areasum} and (\ref{terhiany}) imply that
\begin{equation*}
\sqrt[3]{ \A (S_3)} \geq \sqrt[3]{1/2}\; \frac 1 3 \, r^2.
\end{equation*}
Let now $\La^i$ denote the length of a longest chain in $S_i$ for
$i=1,2,3$. For $i=1$ and $2$, $\La^i$ has the same distribution as
$L_{t_i,n}$ (and as $L^i$). Therefore $\E L^i = \E L_{t_i,n}= \E
\La^i$ for $i=1,2$. Further, $\La^1+\La^2+\La^3 \le L_n$ follows
from concatenating the longest convex chains in the triangles
$S_i$. Thus we have
\begin{equation}\label{osszegkicsi}
\E L^1 + \E L^2 +\E \La^3 = \sum_{i=1}^{3}\E \La^i \leq  \E L_n.
\end{equation}
The random variable $|X_n\cap S_3|$ has binomial distribution with
mean $2A(S_3)n$ which is at least $\kappa = (1/3)^3 r^6 n \ge
(1/3)^3n^{1/2}$. Set $N=\kappa - \sqrt {\kappa \log \kappa}$. Thus
we obtain that for all large enough $n$,
$$N > 0.99 \, \kappa = \frac{0.99}{27} r^6 n,$$
and $N$ tends to infinity with $n$.
Using the estimates (\ref{hoeffding}) and (\ref{also}), again for
large $n$ we have
\begin{align*}
 \E \La^3  &\geq \PP (|X_n\cap S_3| \geq N) \, \E L_N \geq (1-\kappa^{-1/2}) \ 1.57
\, N^{1/3}\\
&\geq 1.569 N^{1/3} \geq 0.52\, r^2 \sqrt[3]n.
\end{align*}
Hence, by (\ref{osszegkicsi})
\begin{equation*}
\E L^1 + \E L^2 \leq \E L_n- 0.52 \,r^2\sqrt[3]n.
\end{equation*}
\end{proof}

\section{Limit shape}\label{lshape}

After the preparations in the previous sections we finally prove
Theorem \ref{limsh}, that is, all chains in $\cc$ lie in a small
neighbourhood of $\G$ with high probability. Note that similar
limit shape results have been proved for convex chains
\cite{BRSZ}; however, they are of different character than the
present case.

We fix the constant $\g\ge 1$. Every result in this chapter holds
for large enough $n$, depending only on $\g$.  We will not always
mention this.

For this proof we set $b=\g n^{1/6}\sqrt{\log n}$. The strong
concentration result of Theorem~\ref{conc1} directly shows that
\[
\PP(L_n <\E L_n -b) \le n^{-\g^2/14}.
\]
We call a convex chain $Y \subset X_n$ {\em long} if its length is
at least $\E L_n - b$.

We will show that all long convex chains lie between the parabolas
$\G_{\rho}$ and $\G_{-\rho}$ with high probability, where high means
$>1-n^{-\g^2/14}$. In view of Proposition~\ref{pr:dist} this
suffices for the proof.

Let $S$ be the triangle with vertices $(0,0.1),(0,0),(0.1,0)$, and
define $H$ to be the event that there is a long convex chain $Y
\subset X_n$ having a point in $S$. We prove first the following
simple fact.

\begin{lemma}\label{l:H}
For $n$ large enough,
$$\PP(H) \le
n^{-\g^2/6}.$$
\end{lemma}

\begin{proof}[\bf{Proof.}] Let $Y$ be a long convex chain with a point in $S$,
and let $y$ be a point of $Y$ where the tangent to $C(Y)$ has
slope $1$. Clearly $y \in S$. Let $Y_1$ be the part
of $Y$ between $(0,1)$ and $y$, and $Y_2$ be the part between $y$ and $(1,0)$. Then
$Y_1$
resp. $Y_2$ are convex chains in the triangle $S_1=\conv
\{(0,1),(0,0),(0.1,0)\}$ and $S_2=\conv \{(0,0.1),(0,0),(1,0)\}$.
As $Y$ is a long convex chain,
\[
\E L_n-b\leq |Y|\le |Y_1|+|Y_2|\le L^1+L^2,
\] where $L^i$ denotes the length of the maximal chain in $S_i$
($i=1,2$), $|Y_i|\le L^i$. As $n \to \infty$, the limit of
$n^{-1/3}\E L_n$ resp. $n^{-1/3}\E L^i$ is $\al$ and
$\al\sqrt[3]{0.1}$. This follows from Theorem~\ref{limit} and
Lemma~\ref{meansub}. So $\lim n^{-1/3}(\E L_n-\E L^1-\E L^2)=\al
(1-2\sqrt[3]{0.1})>1/10$, implying that for large enough $n$
\[
\E L_n-\E L^1-\E L^2> \frac 1{10}\sqrt[3]n >3b= 3\g
n^{1/6}\sqrt{\log n}.
\]
So we have
\begin{align*}
&\PP(H)\le \PP(L^1+L^2> \E L_n-b)\\
&=\PP(L^1+L^2>\E L^1+\E L^2+(\E  L_n-\E L^1-\E L^2)-b)\\
&\leq\PP(L^1+L^2>\E L^1+\E L^2+2b)
\leq \sum_{i=1,2}\PP(L^i>\E L^i+b).
\end{align*}
The triangle $S_i$ is of area $1/20$ so Theorem~\ref{concsub}
shows that
\begin{align*}
&\PP(L^i>\E L^i+b)=\PP(L^i>\E L^i+\g n^{1/6}\sqrt{\log
n})\\
&\leq\PP(L^i>\E L^i+\g\,20^{1/6}(n/20)^{1/6}\sqrt{\log
n/20})\\
&\leq\left(\frac n{20}\right)^{-\g^2 \, 20^{1/3}/14} \leq \frac12\, n^{-\g^2 /6}.
\qedhere
\end{align*}
\end{proof}

After this first step, we estimate the probability of the
existence of a long convex chain not lying between $\G_{-\rho}$
and $\G_{\rho}$. First, we deal with the case when the chain goes
below this region.

We define a set of parabolas. Let $\tr=n^{-1/3}\sqrt{\log n}$, $r_i=-\rho-i \tr$, and
\begin{equation}
G_i=\G_{r_i} \mbox{ where } i=-1,0,1,\dots,g.
\end{equation}
Note that $r_i <0 $. Here we define $g$ by the conditions $G_g
\subset S$ but $G_{i-1}$ is not contained in $S$. Thus the case
when a long chain goes below $G_g$ is covered by Lemma~\ref{l:H}.
Clearly $g$ is limited by $-1< r_g=-\rho-g\tr \ge -1+1/10$. Thus
$g \le n^{2/3}$, say.

The convex polygonal chains $C(Y)$ can be considered as functions
defined on $[0,1]$. We extend the definition of $\G_r$ as 0 on the
interval $[1+r,1]$ if $r<0$, and consider this new ``parabola''
$\G_r$ as a function defined on $[0,1]$. A parabola is said to be
{\em below}, resp. {\em above} $C(Y)$ if the corresponding
function is smaller (larger) than the one corresponding to $C(Y)$.

The following lemma is important.

\begin{lemma}\label{l:finite} There are points $q_{i,j}\in
G_{i-1}, \; j=1,2,\dots,J(i)$ with $J(i) \le n^{1/3}$, such that
the upper envelope of the tangent lines $\ell(q_{i,j})$ of
$G_{i-1}$ at $q_{i,j}$ is a broken polygonal path lying above
$G_i$.
\end{lemma}

\begin{proof}[\bf{Proof.}]The line $\ell_q$, which is tangent to $G_{i-1}$ at $q
\in G_{i-1}$, intersects the graph of $G_i$ in two points. Let
$\la_q$ denote the segment connecting these two points. It is not
hard to check that the length of the segment, $|\la(q)|$,
decreases as $q$ moves away from the center point of $G_{i-1}$. A
simple computation reveals that
\begin{equation}\label{eq:laq}
4\tr \frac{(1+r_i)^2}{(1+r_{i-1})^2} \le |\la_q| \le \sqrt{2\tr(1+r_i)},
\end{equation}
where $q$ only moves up to the point when both endpoints of
$\la(q)$ lie in $G_i$.

Now choose $q_{i,1}$ on $G_{i-1}$ so that the lower endpoint of
$\la(q_{i,1})$ is the intersection of $G_i$ with the $x$-axis.
Once $q_{i,j}$ has been defined, we let $q_{i,j+1}$ be the point
in $G_{i-1}$ for which the lower endpoint of $\la(q_{i,j+1})$
coincides with the upper endpoint of $\la(q_{i,j})$, see
Figure~\ref{parab3} a).  The length of $\G_i$ is smaller than
$2(1+r_i)$. So the process of choosing the $q_{i,j}$ stops after
\[
|J(i)| \le \frac {2(1+r_i)(1+r_{i-1})^2}{4\tr(1+r_i)^2} \le \frac
{(1+r_{i-1})^2}{2\tr(1+r_i)}\le n^{1/3}
\]
steps. This finishes the construction of the points $q_{i,j}$.
The upper envelope of the tangent lines $\ell(q_{i,j})$ is a
convex polygonal path that lies between $G_i$ and $G_{i-1}$ with
edges $\la(q_{i,j})$.
\end{proof}

\begin{figure}[h]
\epsfxsize =\textwidth \centerline{\epsffile{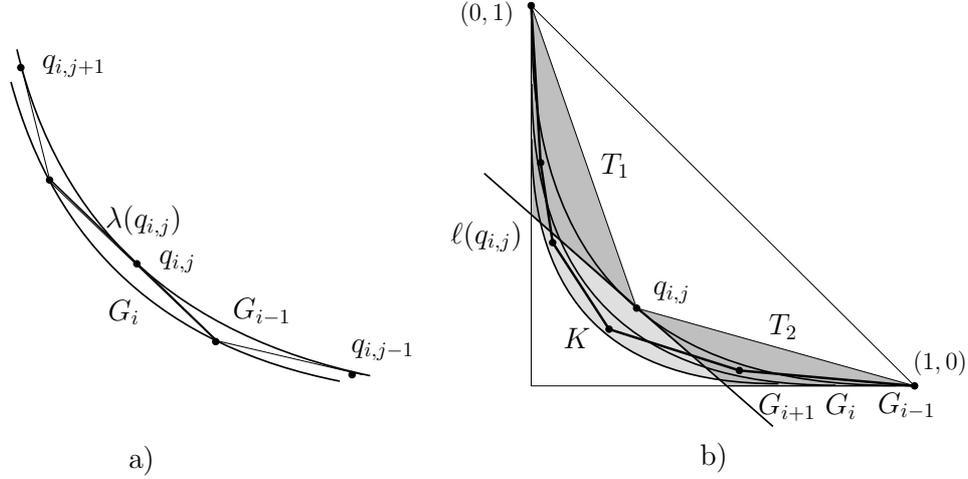}}
  \caption{Long chains below $\Gamma$}
  \label{parab3}
\end{figure}

Now we define $G_i^*$ to be the event that there is a long convex
chain $Y \subset X_n$ with $G_{i+1}$ below $C(Y)$ but $G_i$ not
below $C(Y)$, $i=0,1,\dots,g-1$.

We split these events further. Let $G_{i,j}^*$ be the event that
there is a long convex chain $Y$ with the parabola $G_{i+1}$ below
$C(Y)$ but the line $\ell(q_{i,j})$ not below $C(Y)$; here
$q_{i,j} \in G_{i-1}$ comes from Lemma~\ref{l:finite}. This
implies that $G_i^* \subset \bigcup_{j\in J(i)} G_{i,j}^*$.

\begin{lemma}\label{l:Gij}For every $i=0,\dots,g-1$ and every
$j=1,\dots J(i)$, $\PP(G_{i,j}^*)\le 3 n^{-8\g^2/7}$.
\end{lemma}

Before the proof we state (and prove) the following corollary.

\begin{corollary}\label{cor1}
The probability that there is a long convex chain $Y \subset X_n$
such that $C(Y)$ is not above $\G_{-\rho}$ is at most
$n^{-\g^2/6}+3 n^{-\g^2/7}$.
\end{corollary}

This is quite easy: If there is such a chain, then either $H$, or
some $G_i^*$ ($i=0,1,\dots,g-1$) occur. Since $G_i^* \subset
\bigcup _{j\in J(i)}G_{i,j}^*$, $g J(i) \leq n$ and $\g \geq 1$,
the corollary follows from Lemmas~\ref{l:Gij} and \ref{l:H}.

\begin{proof}[{\bf Proof} {\rm of Lemma~\ref{l:Gij}.}] Let $T_1,T_2$
be the two triangles determined by $q_{i,j}$ and $\ell(q_{i,j})$
as usual, and let $K=K_{i,j}$ be the convex set between
$\la(q_{i,j})$ and $G_{i+1}$, see Figure~\ref{parab3} b).

We estimate $A(K)$ as follows. A simple calculation as in
(\ref{eq:laq}) yields that the diameter of $K$ is at most $2
\sqrt{\tr}$, and $K$ is between the line $\ell(q_{i,j})$ and the
parallel line tangent to $\G_{i+1}$. The distance of these lines
is at most $2 \tr/\sqrt 8$ as one can easily check using
(\ref{eq:dist}). Then $A(K)\le \sqrt{2} \tr^{3/2}$.

A long convex chain $Y\subset X_n$ which is above $G_{i+1}$ but
not above $\ell(q_{i,j})$ splits into 3 parts: $Y_1=T_1\cap Y$,
$Y_2=T_2\cap Y$, and $Y_3=K\cap Y$. Here $Y_1,Y_2$ are convex
chains in $T_1$ (from $(0,1)$ to $q_{i,j}$) and in $T_2$ (from
$q_{i,j}$ to $(1,0)$), and $Y_3$ is in convex position in $K$. So
with the notations of the previous section we have
\[
|Y_1| \le L^1,\; |Y_2| \le L^2, \mbox{ and } |Y_3| \le M(K,n).
\]
Since $Y$ is a long convex chain, $|Y_1|+|Y_2|+|Y_3| \ge \E L_n-b$.
This implies that $L^1+L^2+M(K,n) \ge \E L_n-b$. We are going to
show that this event has small probability.

We apply Lemma~\ref{l:M>m} with $\mu=b/5$. For
large enough $n$ it implies that
\begin{equation}\label{eq:last}
\PP(M(K,n)\ge b/5)<(b/5)^3 2^{-b/5}+n2^{- b^3/(480e5^3)}<
2^{-n^{1/6}}< n^{-8\gamma^2 /7},
\end{equation}
since the condition
$1920\, e^2\,A(K)n \le (b/5)^3$ is satisfied as $A(K) \le
 \sqrt{2} \tr^{3/2}< \sqrt{2} n^{-1/2}(\log n)^{3/4}$ and $(b/5)^3=\g^3 n^{1/2}
(\log n)^{3/2}/125$.

Next,
\begin{align}\label{eq:start}
& \PP(L^1+L^2+M(K,n) \ge \E L_n -b)\nonumber \\
&\le \PP(L^1+L^2 \ge \E L_n -1.2 \, b)+\PP(M(K,n) \ge b/5) \\
&\le \PP(L^1+L^2\ge \E L_n -1.2\,b)+n^{-8\g^2/7}.\nonumber
\end{align}
Now Lemma~\ref{l:2chains} implies that $\E L^1+\E L^2 \le \E L_n -0.52
r_{i-1}^2\sqrt[3]n$, and hence
\begin{align}\label{eq:cont}
&\PP(L^1+L^2\ge \E L_n -1.2\,b)\nonumber \\
&\le \PP(L^1+L^2\ge \E L^1+\E L^2+0.52
r_{i-1}^2\sqrt[3]n -1.2\,b)\\
&\le \sum_{i=1,2}\PP(L^i\ge \E L^i+0.26 r_{i-1}^2\sqrt[3]n -0.6\, b)\nonumber \\
&\le \sum_{i=1,2}\PP(L^i\ge \E L^i+4b).\nonumber
\end{align}
Here the last step is justified by observing that $r_{i-1}\le
r_{-1}=-\rho+\tr$ and so for large enough $n$
\begin{align}\label{eq:folytat}
0.26\, r_i^2\sqrt[3]n &\geq  0.26\, n^{1/3}\big(3 \sqrt{2} \g^{1/2}n^{-1/12}(\log
n)^{1/4}  - n^{-1/3}\sqrt{\log n}  \big)^2 \nonumber \\
&> 4.6 \g\, n^{1/6}\sqrt{\log n}=4.6 \, b.
\end{align}

Next, we estimate $\PP(L^i\ge \E L^i+4b)$.
When $t_i = 2 A(T_i) \ge n^{-5/6}$, we use Theorem~\ref{concsub} with
$\tau=5/6$:
\begin{align*}
& \PP(L^i\ge \E L^i+4b) = \PP(L^i \geq \E  L^i +
4 \gamma \sqrt{\log n} \; n^{1/6} )\\
&\leq \PP(L^i
\geq \E L^i + 4\gamma \sqrt{\log n / \log (n t_i)} \, \sqrt{\log (n
t_i)} \; (n t_i)^{1/6} )
\\
&\leq (n t_i)^{-\gamma^2 8\log n / 7 \log (n t_i)}=
n^{-8\gamma^2/7}.
\end{align*}
The last inequality holds because of the Remark following
Theorem~\ref{concsub}, since
\[
1 \leq 4 \gamma \sqrt{\log n / \log (n t_i)} \leq \gamma \, 4
\sqrt{6}.
\]

Finally, when $t_i < n^{-5/6}$, the expected number of points in
$T_i$ is $t_i n < n^{1/6}$. So for the random variable $|T_i\cap
X_n|$ inequality (\ref{hoeffdingfelso}) implies that
\begin{align*}
\PP \big(\, |T_i\cap X_n| \geq  4 \gamma \sqrt{\log n} \; n^{1/6} \big)&\leq
\left(\frac{e \, t_i n}{4 \gamma \sqrt{\log n} \; n^{1/6}}   \right)^{4 \gamma
\sqrt{\log n} \; n^{1/6}}\\
\leq \left(\frac{e }{4 \gamma \sqrt{\log n}}   \right)^{n^{1/6}}&< n^{-8\gamma^2/7}
\end{align*}
for large enough $n$,  and hence
\[
\PP \big(L^i \geq \E L^i + 4 \gamma \sqrt{\log n} \; n^{1/6} \big) <
n^{-8\gamma^2/7}.
\]
Thus $\PP \big(L^i \geq \E L^i +  4b
\big)\le n^{-8\gamma^2/7}$ for $i=1,2$ in all cases.
\end{proof}

Now we handle the case of parabolas going above $\G_\rho$. Set
$R_i=\rho+i\de$ where $\de=n^{-1/2}\sqrt {\log n}$. We define
another series of parabolas:
\begin{equation}
\Ge_i=\G_{R_i}, \; i=-1,0,1,\dots,f
\end{equation}
where $f$ is limited by $\rho+f\de \le 3$. Thus $f \le n^{1/2}$,
say.

The following geometric lemma is similar to Lemma~\ref{l:finite}.

\begin{lemma}\label{l:finite+} There are points $p_{i,j}\in
\Ge_{i-1}, \; j=1,2,\dots,\J(i)$ with $\J(i)\le n^{1/2}$ such that
the following holds. For each convex chain $Y\subset X_n$
with $\Ge_{i+1}$ above $C(Y)$ but $\Ge_i$ not above $C(Y)$, there is a
$p_{i,j}$ such that the line $\ell(p_{i,j})$ is below $C(Y)$.
\end{lemma}

\begin{proof}[\bf{Proof.}]For such a long chain $Y$ there is a smallest
$R>\rho$ with $\G_R$ above $C(Y)$. Then $C(Y)$ and
$\G_R$ have a common point and a common tangent $\ell$ at
that point (because both $C(Y)$ and $\G_R$ are convex). Let
$p$ be the point on $\Ge_i$ such that the line $\ell(p)$, tangent at
$p$ to $\Ge_i$, is parallel with $\ell$. It is evident that $C(Y)$
is above $\ell(p)$.

Let $L$ denote the set of lines that are tangent to $\Ge_i$ and
that have both $(0,0)$ and $(1,1)$ above it. We will construct a
set of points $p_{i,j}\in \Ge_{i-1}$ such that each line in $L$ is
above the segment $\ell(p_{i,j}) \cap\, T$ for some
$j=1,2,\dots,\J(i)$. This construction then guarantees what the
lemma requires.

We need one more piece of notation. Given $p_{i,j}$ let
$[A_j,B_j]$ be the segment $T\cap \ell(p_{i,j})$, with $A_j$ on
the $x$-axis and $B_j$ on the $y$-axis. We shall construct the
sequence of the $A_j$'s and $B_j$'s.

The construction starts with $p_{i,1}$ at the midpoint of
$\Ge_{i-1}$ and we define first the other $p_{i,j}$ with $A_1$
closer to the origin than $A_j$. See Figure~\ref{parab4}~a).
Assume $p_{i,j}$ has been found. There is a unique tangent,
$\ell$, to $\Ge_i$ passing through $B_j$. Let $A_{j+1}$ be the
intersection point of $\ell$ with the $x$-axis, and $p_{i,j+1}$
the common point of $\Ge_{i-1}$ with the tangent to $\Ge_{i-1}$
through $A_{j+1}$. The construction is finished when we reach
$x(A_j)<0$, here $x(A_j)$ denotes the $x$-coordinate of $A_j$.
Corollary~\ref{erinto} implies that
\[
|A_j A_{j+1}|= |B_j B_{j+1}|= (1+R_i) - (1+R_{i-1})=\de.
\]
Since $x(A_1)<1/2$, we reach $x(A_j)<0$ after at most $(2\de)^{-1}$ steps.

The construction satisfies what we need: if a tangent to $\Ge_i$
intersects the triangle in the segment $[A,B]$ with $A$ on the $x$
axis and $x(A)\in [0,1/2]$, then $A$ is between $A_{j+1}$ and
$A_j$ for some $j$, and the segment $[A,B]$ is above the segment
$\ell(p_{i,j})\cap T$.

The construction is extended to the other half of $\Ge_{i-1}$
symmetrically, and $\J(i) \le 2(2\de)^{-1}\le n^{1/2}$ follows.
\end{proof}

\begin{figure}[h]
\epsfxsize = \textwidth \centerline{\epsffile{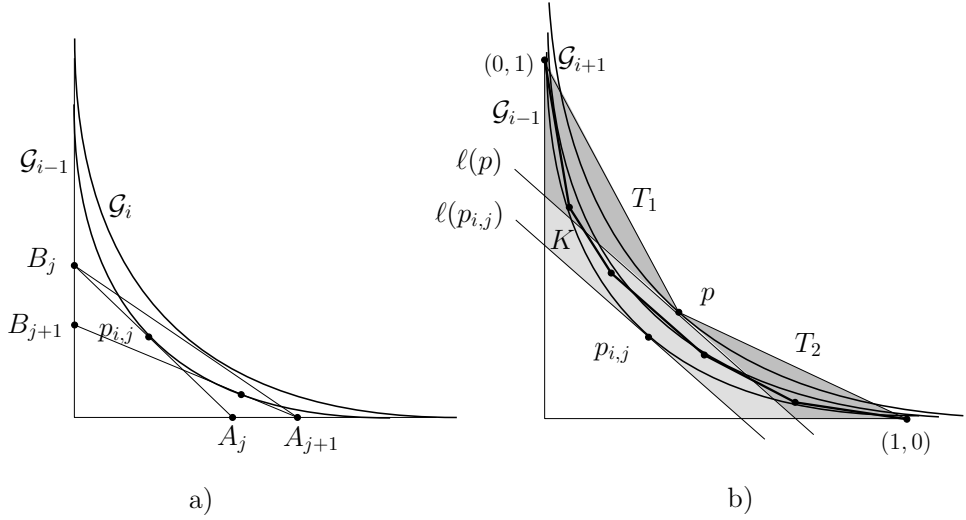}}
  \caption{Long chains reaching above $\Gamma$}
  \label{parab4}
\end{figure}

Next we define $\Ge_i^*$ ($i=0,1,\dots,f-1$) to be the event that
there is a long convex chain $Y\subset X_n$ such that $\Ge_{i+1}$
is above $C(Y)$ but $\Ge_i$ is not above $C(Y)$,
$i=0,1,\dots,f-1$. Further, let $\Ge_{i,j}^*$ be the event there
is a long convex chain $Y \subset X_n$ with $C(Y)$ below
$\Ge_{i+1}$ but not below $\ell(p_{i,j})$ (remember that $p_{i,j}
\in \Ge_{i-1}$). Here $i=0,\dots,f-1$ and $j=1,\dots,\J(i)$. We
have now the following result, similar to Lemma~\ref{l:Gij}.

\begin{lemma}\label{l:Fij}For every $i=0,\dots,f-1$ and every
$j=1,\dots \J(i)$, $\PP(\Ge_{i,j}^*)\le 3n^{-8\g ^2/7}$.
\end{lemma}

This lemma immediately implies the following corollary.

\begin{corollary}\label{cor2}
The probability that there is a long convex chain $Y \subset X_n$
such that $C(Y)$ is not below $\G_{\rho}$ is at most
$3n^{-\g^2/7}$.
\end{corollary}

The proof follows from the facts that $\Ge_i^* \subset
\bigcup_{j\in \J(i)}\Ge_{i,j}^*$, $f \le n^{1/2}$, $\J(i)\le
n^{1/2}$, and $\g \geq 1$. Now we give the proof of Lemma~\ref{l:Gij} which is
analogous to that of Lemma~\ref{l:Gij}.

\begin{proof}[{\bf Proof }{\rm of Lemma~\ref{l:Fij}}.]
Let $\ell(p)$ be the unique tangent to $G_{i+1}$ which is parallel
with $\ell(p_{i,j})$, and $p$ be the common point of $\ell(p)$ and
$\G_{i+1}$, see Figure~\ref{parab4}~b). Let $T_1,T_2$ be the two
triangles determined by $p$ and $\ell(p)$, and let $K=K_{i,j}$ be
the part of $T$ that lies between $\ell(p_{i,j})$ and $\ell(p)$.
Since the distance of these two lines is at most $2\de/\sqrt 8$,
$A(K) \le \de$.

A long convex chain $Y\subset X_n$ which is below $\Ge_{i+1}$ but
not below $\ell(p_{i,j})$ splits into 3 parts: $Y_1=T_1\cap Y$,
$Y_2=T_2\cap Y$, and $Y_3=K\cap Y$. Here $Y_1,Y_2$ are convex
chains in $T_1$ (from $(0,1)$ to $p$) and in $T_2$ (from
$p$ to $(1,0)$), and $Y_3$ is in convex position in $K$. So
\[
|Y_1| \le L^1,\; |Y_2| \le L^2, \mbox{ and } |Y_3| \le M(K,n).
\]
Since $Y$ is a long convex chain, $|Y_1|+|Y_2|+|Y_3| \ge |Y|\ge \E
L_n-b$, and so $L^1+L^2+M(K,n) \ge \E L_n-b$. We are going to show
that this event has small probability.

We apply Lemma~\ref{l:M>m} again with $\mu=b/5$. For sufficiently large $n$ the
condition
$1920\, e^2\,A(K)n \le (b/5)^3$ is satisfied, since $A(K) \le \de
= n^{-1/2}\sqrt{\log n}$ and $(b/5)^3=\g^3 n^{1/2} (\log n)^{3/2}/125$. So we have,
just as in (\ref{eq:last}),
\[
\PP(M(K,n)\ge b/5)< n^{-8\gamma^2 /7}.
\]
Therefore the estimate (\ref{eq:start}) applies without change:
\begin{equation*} \PP(L^1+L^2+M(K,n) \ge \E L_n -b)
\le \PP(L^1+L^2\ge  \E L_n -1.2 \, b)+n^{-8 \g^2 /7}.
\end{equation*}
Now Lemma~\ref{l:2chains} implies that $\E L^1+\E L^2 \le \E L_n -0.52
R_{i+1}^2\sqrt[3]n$, and just as in (\ref{eq:cont}),
\begin{align*}\label{eq:cont}
\PP(L^1 + L^2 \geq \E L_n - 1.2 \, b)
&\leq \sum_{i=1,2}\PP(L^i\ge \E L^i+0.26 R_{i+1}^2\sqrt[3]n -0.6 \, b)\\
&\leq \sum_{i=1,2}\PP(L^i\ge \E L^i+4b).
\end{align*}
Here the last step is justified just as in (\ref{eq:folytat})
except that this time $R_{i+1}\ge R_{1}=\rho+\de$. Finally, we
bound $\PP(L^i\ge \E L^i+4b)$ the same way as in the proof of
Lemma~\ref{l:Gij} to obtain
\[
\PP(L^i\ge \E L^i+4 b )\le n^{-8\g^2/7}.
\qedhere
\]
\end{proof}

\begin{proof}[{\bf Proof }{\rm of Theorem~\ref{limsh}}.]
Considering Proposition~\ref{pr:dist}, we have to estimate the
probability that there is a {\em longest} convex chain not lying
between $\G_{-\rho}$ and  $\G_{\rho}$. This event splits into two
parts: either the longest convex chain is not long, or there is a
long convex chain not between $\G_{-\rho}$ and  $\G_{\rho}$. The
probability of the first event is estimated by
Theorem~\ref{conc1}, while the second part is handled via
Corollaries \ref{cor1} and \ref{cor2}. Therefore the probability
in question is at most
\begin{equation*}
n^{-\g^2/14}+  n^{-\g^2/6}+ 6 n^{-\g^2/7}<2n^{-\g^2/14}. \qedhere
\end{equation*}
\end{proof}

{\bf Remarks. } In this proof one can avoid using the estimate on
$M(K,\mu)$. In fact, choosing $\delta$ and $\triangle$ small
enough, the set $K$ contains more than $b/5$ points of $X_n$ with
very small probability. So, with high probability, it cannot add
much to the size of a long convex chain. There are more events
$G_i^*$ and $G_{i,j}^*$, which has a minor effect on the final
result. Also, the triangle $S$ in Lemma~\ref{l:H} is to be chosen
much smaller.

An important step in our proof is Lemma~\ref{l:2chains},
essentially implying that if the distance between $\G$ and the
farthest point of a convex chain from $\G$ is ``large'', then the
chain cannot be too long. Conditioning on the location of this
farthest point would allow an elegant conditional expectation
argument. However, fixing the farthest point modifies the
underlying probability space and therefore the estimate coming
from Lemma~\ref{l:2chains} is no longer valid. To eliminate this
difficulty, we chose to define finitely many subcases and estimate
them separately, which can also be considered as a finite
approximation of the continuous conditional expectation.

\section{Numerical experiments}\label{experi}

In the final section we summarize the observations obtained by
computer simulations.

The search for the longest convex chains can be accomplished by an
algorithm which has running time $O(n^2)$. This algorithm works as
follows. We order the points by increasing $x$ coordinate, and then
recursively create a list at each point. The $k$th element on the
list at point $p$ contains the minimal slope of the last segment of
chains starting at $p_0$ and ending at $p$ whose length is exactly
$k$, and a pointer to the other endpoint of this last segment. For
creating the list at the next point $p$, we have to search the
points before $p$, and see if $p$ can be added to the chains while
preserving convexity.

This algorithm can be speeded up with some (not fully justified
but useful) tricks. First of all, Theorem~\ref{limsh} guarantees
that we have to search only among the points close to $\G$. The
simulations show that most longest convex chains are located in a
small neighbourhood of $\G$, whose radius is in fact of order
approximately $n^{-1/3}$, much smaller than the width of order
$n^{-1/12}$ given by Theorem~\ref{limsh}. Therefore the search can
be restricted to a subset of the points with cardinality of order
$n^{2/3}$. Second, when looking for the longest chain, we have to
search only points relatively close to $p$, and chains which are
already relatively long, thus reducing memory demands.
\begin{table}[!h] \label{tabla}
\newcommand\TT{\rule{0pt}{2.6ex}}
\newcommand\BB{\rule[-1.2ex]{0pt}{0pt}}
\begin{tabular}{c| c | c | c | c }
$n$ & $n^{-1/3} \E L_n$ \BB & $d_n$
&Distance$ /\sqrt{2}$ & Deviation
\\
\hline
1000 \TT & 2.532 & 4 & 0.270 & 1.254\\
10000& 2.768 & 5 & 0.200 & 1.383\\
15625& 2.813 & 5 & 0.150 & 1.293\\
50000& 2.885 & 5 & 0.100 & 1.411\\
75000& 2.906 & 5 & 0.070 & 1.580\\
100000& 2.917& 5 & 0.060 & 1.431\\
125000& 2.926 & 5 & 0.050 & 1.637\\
421875& 2.959 & 5 & 0.012 & 1.732\\
1000000 \BB & 2.976 & 6 & 0.012 & 2.023\\
\end{tabular}
\\[10pt]
\caption{Results obtained by the simulation}
\end{table}

With the above method, the search can be executed for up to
$5\cdot 10^4$ active points, in which case examining one sample
takes about 2 minutes. As the experiments show, this provides a
good approximation for $n$'s up to order $10^6$. In each
experiment, we increased the width of the searched neighbourhood
until the increment did not generate a significant change in the
average length of the longest convex chain. The results obtained
by this method, although giving only a lower bound for $\E L_n$,
are heuristically close to it. $\alpha=3$.
\begin{figure}[h!]
\epsfxsize =0.4\textwidth \centerline{\epsffile{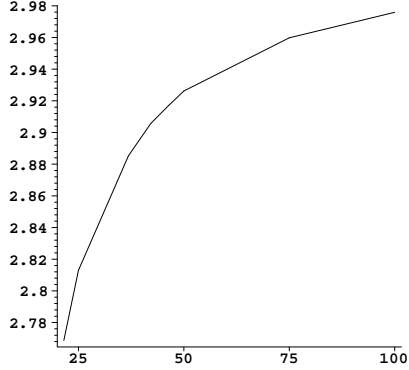}}
  \caption{Results for $n^{-1/3} \E L_n$, illustrated as a function of $n^{1/3}$. }
  \label{varh}
\end{figure}

Our largest search has been done for $n= 10^6$. The number of
samples was $250$ except for the cases $n=25^3$ and $n=10^6$, where
we used $500$ samples in order to model the distribution of $L_n$
(see Figure~\ref{eloszl}).

The results below well illustrate what the proof of
Theorem~\ref{limit} suggests, namely, that $ n^{-1/3} \E L_n$ is
increasing with $n$. Also, the data seem to confirm that $\alpha=3$.
\begin{figure}
\centering
\includegraphics[width=0.4 \textwidth]{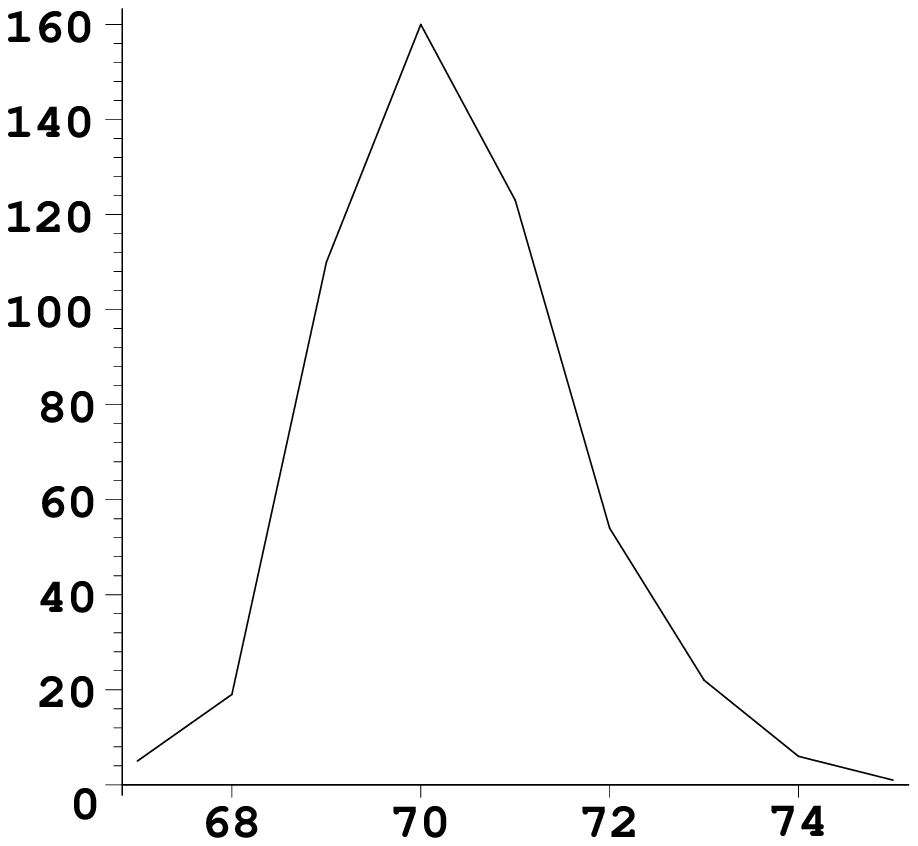}
\hspace{1 cm}
\includegraphics[width=0.4 \textwidth]{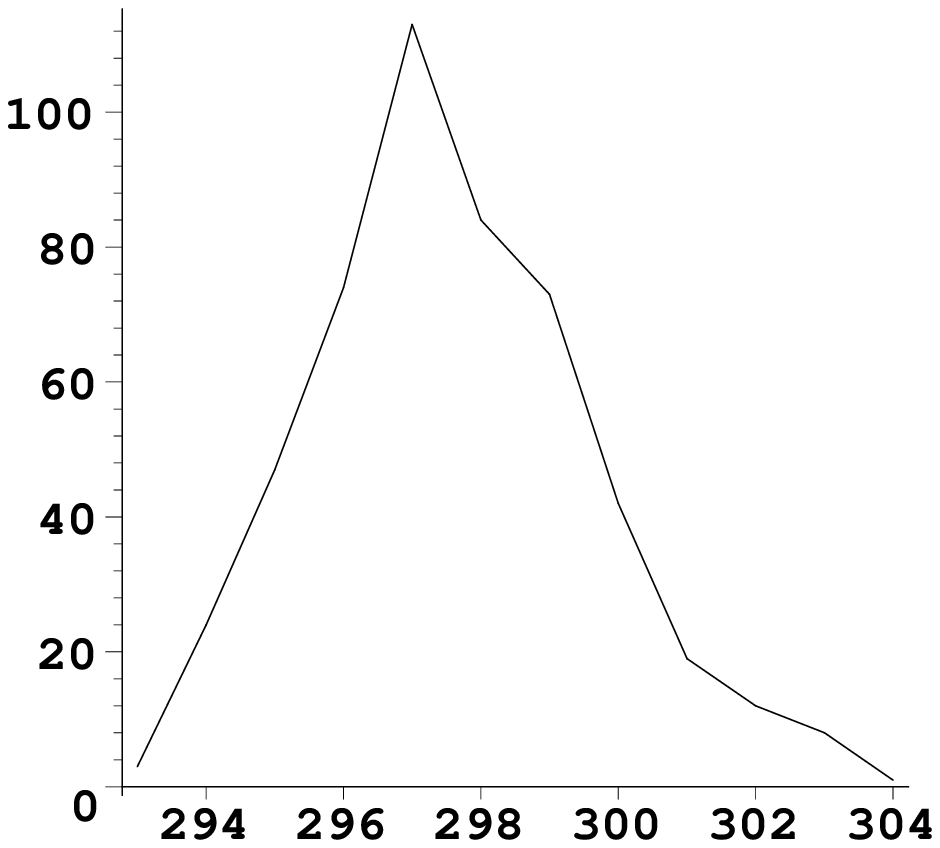}
\caption{Distribution of  $L_n$, 500 samples, $n=25^3$ and
$n=10^6$.} \label{eloszl}
\end{figure}

On Table 1 we list the results obtained by the program. The first
column is the number of points chosen in $T$, the second is the
average of  $n^{-1/3}L_n$. The third column contains the half-length
of the interval of the values of $L_n$, that is, $d_n= \lfloor \max
|L_n - \E L_n|\rfloor $. This is noticeably small even for $n=10^6$.
In the fourth column we list $1/\sqrt{2}$ times the radius of the
neighbourhood of parabola we used for the search (the term
$\sqrt{2}$ comes from a transformation of coordinates). The last
data are the standard deviation of the set of values of $L_n$, ie.
the square-root of its variance.

Figure~\ref{varh} illustrates the linear interpolation of $n^{-1/3}
\E L_n $ as a function of $n^{1/3}$. It is based on the data shown
on Table 1.

As we know from Theorem~\ref{conc1}, $L_n$ is highly concentrated
near its expectation. This phenomenon is well recognizable on
Figure~\ref{eloszl}, where we plot the distribution in the cases
$n=25^3(=15625)$ and $n=10^6$ with $500$ samples.

\section{Acknowledgements}
We express our special thanks to G\'abor Tusn\'ady for his
constant attention and interest in this piece of work, for
valuable ideas concerning computer simulations, and in particular
for pointing out an error in the earlier version of this paper. We
also thank Zolt\'an Kov\'acs for his suggestions regarding the
implementation of the program. The second author was supported by
Hungarian National Foundation Grants T 60427 and T 62321. Finally,
we dedicate this piece of work to the memory of the late Professor
S\'andor Cs\"org\H o, whose zest for life and enthusiasm for
mathematics will always be a constant inspiration to us.

\vspace{0.5 cm}

\begin{minipage}[t]{0.5 \textwidth}
{\sc Gergely Ambrus}
\\
  {\footnotesize Department of Mathematics\\[-1.5mm]
  University College London\\[-1.5mm]
  Gower Street, London WC1E 6BT\\[-1.5mm]
  England, U.K. \\[-1.5mm]
  and\\[-1.5mm]
  Bolyai Institute\\[-1.5mm]
  University of Szeged\\[-1.5mm]
  Aradi v\'ert. tere 1, 6720 Szeged\\[-1.5mm]
  Hungary\\[-1.5mm]
  e-mail: {\tt g.ambrus@ucl.ac.uk}
       }
\end{minipage}
\begin{minipage}[t]{0.5 \textwidth}
 {\sc Imre B\'ar\'any}
\\
  {\footnotesize
  R\'enyi Institute of Mathematics\\[-1.5mm]
  Hungarian Academy of Sciences\\[-1.5mm]
  PO Box 127, 1364 Budapest\\[-1.5mm]
  Hungary\\[-1.5mm]
  and\\[-1.5mm]
  Department of Mathematics\\[-1.5mm]
  University College London\\[-1.5mm]
  Gower Street, London WC1E 6BT\\[-1.5mm]
  England, U.K. \\[-1.5mm]
  e-mail: {\tt barany@renyi.hu}
       }
\end{minipage}
\end{document}